\documentclass[11pt]{article}
\usepackage{indentfirst}
\usepackage{amsmath,amssymb,amsfonts,amsthm,graphics}
\usepackage{makeidx}
\usepackage{indentfirst,graphics,epsfig,psfrag}
\usepackage{ifpdf}
\usepackage{color}
\newtheorem{thm}{Theorem}[section]

\newtheorem{lem}{Lemma}[section]

\theoremstyle{definition}

\def\-{\mbox{--}}

\begin{document}
\title{\Large\bf Some upper bounds for 3-rainbow index of graphs }
\author{
\small  Tingting Liu, Yumei Hu\footnote{supported by NSFC No. 11001196. } \\
\small  Department of Mathematics, Tianjin University, Tianjin 300072, P. R. China\\
\small E-mails: ttliu@tju.edu.cn; huyumei@tju.edu.cn;
 }
\date{}
\maketitle
\begin{abstract}
 A  tree $T$, in  an edge-colored graph $G$, is
called {\em a rainbow tree}  if no two edges of $T$ are assigned the
same color. A  {\em $k$-rainbow coloring }of $G$ is an edge coloring
of $G$ having the property that for every set $S$ of $k$ vertices of
$G$, there exists a rainbow tree $T$ in $G$ such that $S\subseteq
V(T)$. The minimum number of colors needed in a $k$-rainbow coloring
of $G$ is the {\em $k$-rainbow  index of $G$ }, denoted by
$rx_k(G)$. In this paper, we consider $3$-rainbow index $rx_3(G)$ of $G$. We first show that for  connected graph $G$ with minimum degree $\delta(G)\geq 3$, the tight upper bound of $rx_3(G)$ is $rx_3(G[D])+4$, where $D$ is the connected $2$-dominating set of $G$. And then we determine a tight upper bound for $K_{s,t}(3\leq s\leq t)$ and a better bound for $(P_5,C_5)$-free graphs. Finally, we obtain a sharp bound for $3$-rainbow index of general graphs.

\end{abstract}

\section{Introduction}

All graphs considered in this paper are simple, connected and
undirected. We follow the terminology and notation of Bondy and
Murty \cite{bondy2008graph}. An edge-colored graph $G$ is {\em
rainbow connected } if any two vertices are connected by a path whose
edges have distinct colors. The {\em rainbow connection
number} $rc(G)$ of $G$, introduced by Chartrand et al.
\cite{Chartrand2008graph}, is the minimum number of colors that
results in a rainbow connected graph $G$.

Later, another generalization of rainbow connection number was introduced by Chartrand
et al.\cite{Chartrand2009graph} in 2009. A tree $T$ is a {\em rainbow tree} if no two edges of $T$ are
colored the same. Let $k$ be a fixed integer with $2\leq k\leq n$.
An edge coloring of $G$ is called a {\em $k$-rainbow coloring} if
for every set $S$ of $k$ vertices of $G$, there exists a rainbow
tree in $G$ containing the vertices of $S$. The {\em $k$-rainbow
index} $rx_k(G)$ of $G$ is the minimum number of colors needed in a
$k$-rainbow coloring of $G$. It is obvious that $rc(G)=rx_2(G)$.

Let $k$ be a positive integer. A subset $D\subseteq V(G)$ is a $k$-dominating set of the graph $G$ if $|N_G(v)\cap D|\geq k$ for every $v\in V\setminus D$.
The $k$-domination number $\gamma_k(G)$ is the minimum cardinality among the $k$-dominating sets of $G$. Note that the $1$-domination number
$\gamma_1(G)$ is the usual domination number
$\gamma(G)$. A subset $S$ is a connected $k$-dominating set if it is a $k$-dominating set and the graph induced by $S$
is connected. The connected $k$-domination number $\gamma_k^c(G)$
represents the cardinalities of a minimum connected $k$-dominating set.
For $k=1$, we write $\gamma_c$ instead of $\gamma_1^c(G)$.

Chakraborty  et al. \cite{chakraborty2009hardness} showed that
computing the rainbow connection number of a graph is NP-hard. So it
is also NP-hard to compute $k$-rainbow index of an arbitrary graph.

Chandran et al.  \cite {Chandran2011dominating}
use  a strengthened connected dominating set (connected $2$-way dominating set) to prove $rc(G)\leq rc(G[D])+3$. This led us to the investigation of what is  strengthening of a connected dominating set which can
apply to consider $3$-rainbow index of a graph.

Recently, for $3$-rainbow index, Li et al.  did some basic results and they obtained the following
theorem.

\begin{thm}\cite{Li2013}\label{thm0}
Let $G$ be a  $2$-connected graph of order $n$ $(n \geq 4)$. Then $rx_3(G)\leq n-2$, with
equality if and only if $G = C_n$ or $G$ is a spanning subgraph of $3$-sun or $G$ is a spanning
subgraph of $K_5\setminus e$  or $G$ is a spanning subgraph of $K_4$.
\end{thm}

Here, a $3$-sun is a graph $G$ which is defined from $C_6$ = $v_1v_2 \cdots v_6v_1$ by adding three edges
$v_2v_4$, $v_2v_6$ and $v_4v_6$.

Chartrand et al. \cite{Chartrand2008graph} obtained that
for integers s and t with $ 2\leq s \leq t$,
$rc(K_{s,t})=min\{\sqrt[s]{t},4\}$. Thus, $rx_2(G)=rc(G)\leq 4$.
Li et al.  \cite{Li2013}
consider the regular complete bipartite graphs $K_{r,r}$. They show $rx_3(K_{r,r})=3$ for integer $r$ with $r\geq 3$.

In this paper, we focus on $3$-rainbow index.
In section $2$, we adopt connected $2$-dominating set to study $3$-rainbow index. A coloring strategy is obtained  which uses only a constant number of extra colors outside the dominating set. We prove that $rx_3(G)\leq rx_3(G[D])+4$, where $D$ is the connected $2$-dominating set of
$G$.  In section $3$, We determine a sharp bound of $3$-rainbow index for $K_{s,t}$ ( $3\leq s\leq t$) and an upper bound for $(P_5,C_5)$-free. In section $4$, we investigate a sharp upper for $rx_3(G)$ of general graphs by block decomposition and an upper bound for graphs with $\delta(G)\geq 3$ by connected $2$-dominating set.

\section{A sharp upper of $3$-rainbow index in terms of connected $2$-dominating set}
\begin{thm}\label{thm5}
Let $G$ be a connected graph with minimal degree
$\delta\geq 3$. If $D$ is a connected $2$-dominating set of $G$, then $rx_3(G)\leq rx_3(G[D])+4$ and the bound is tight.
\end{thm}

\begin{proof}
We prove the theorem by demonstrating that $G$ has a 3-rainbow coloring
with $rx_3(G[D])+4$ colors. For $x\in V(G)\setminus D$, its neighbors in $D$ will be called foots of $x$, and the corresponding edges will be called legs of $x$.

We give $G[D]$ a 3-rainbow coloring using colors
$1, 2, \cdots, k~(k=rx_3(G[D]))$. Let $H:=G\setminus D$.
Partition $V(H)$ into sets $X,Y,Z$
as follows. $Z$ is the set of all isolated vertices of $H$.
In every nonsingleton  connected component of $H$, choose a spanning tree.
So we construct a forest on $W:=V(H)\setminus Z$ and choose $X$ and $Y$ as
any one of the bipartitions defined by this forest. Color every $X-D$ edge with $k+1$ or  $k+2$ where each of $k+1,k+2$ appears at least once, every
$Y-D$ edge with $k+1$ or $k+3$ where each of $k+1,~k+3$ appears at least once, every edge between $X$ and $Y$ with $k+4$. Since $G$  has a minimal degree $\delta\geq 3$, every vertex in $Z$ will have at least three neighbors in $D$. Color two of them with $k+1$ and $k+3$ and all the others with $k+4$. Next, we show that under such an edge coloring for any three vertices in $D$ there exists a rainbow tree containing them.

For three vertices $(x,y,z)\in D\times D\times D$, there is already a
rainbow tree containing them in $G[D]$. For three vertices $(x,y,z)\in$
$D\times D\times V(H)$ (or $D\times V(H)\times V(H)$), join any one leg of
$z$ (or $k+1$, $k+3$ ($k+2$) legs of $y$ and $z$ ) with a rainbow
tree containing the corresponding  foot (or two foots), $x$ and $y$ (or $x$) in $G[D]$.
Now we consider the case three vertices $(x,y,z)\in$ $V(H)\times V(H)\times V(H)$.
For three vertices $(x,y,z)\in Z\times Z\times Z$, join three edges which color $k+1,k+4$ and $k+3$ with a rainbow tree containing the corresponding  foots
$(x',y',z')$ in $D$.
For two vertices $(x,y)\in Z\times Z$, $z \in W$, join a
$k+1$ leg of $z$ and $k+3,~k+4$ legs of $x,~y$ with a rainbow
tree containing the corresponding foots in $G[D]$.
Consider one vertex $x\in Z$, two vertices $(y,z) \in W\times W$.
If $(y,z)\in X\times X$, join a $k+4$ leg of $x$ and
$k+1,k+2$ legs of $y$ and $z$
with a rainbow tree containing the corresponding foots  in $G[D]$.
If $(y,z)\in X\times Y$ or $(y,z)\in Y\times Y$, join a $k+4$ leg
of $x$ and $k+1,k+3$ legs of $y$ and $z$ with a rainbow tree containing the
corresponding foots  in $G[D]$.
Then consider  three vertices $(x,y,z) \in W\times W \times W$. If
$(x,y,z) \in X \times X \times X$, we know, for $x \in X$,
$x$ has a neighbor $y(x) \in Y$. $x-y(x)$ edge (colored $k+4$) and
$k+3$ leg of $y(x)$, join $k+1$ leg  of $y$ and $k+2$ leg of $z$
with a rainbow tree containing the corresponding foots in $G[D]$.
Similarly, in other cases, we can find a rainbow tree containing them.
Hence, $G$ has a 3-rainbow coloring with   $rx_3(G[D])+4$ colors.

The proof of tightness is given in the next section.
\end{proof}
\section{Upper bounds for $3$-rainbow index of some special graphs}

In this section, we consider two special graphs: complete bipartite graphs $K_{s,t}$ and $(P_5,C_5)$-free graphs.

\begin{thm}\label{thm6}
For any  complete bipartite graphs $K_{s,t}$ with $3\leq s \leq t$,
$rx_3(K_{s,t})\leq min \{6,s+t-3\}$, and the bound is tight.
\end{thm}
\begin{proof}
 Because
$K_{s,t}$ with $3\leq s \leq t$ is a 2-connected graph, by Theorem \ref{thm0}, we have, $rx_3(K_{s,t})\leq s+t-3$. The equality clearly holds for $s=t=3$ since
$rx_3(K_{3,3})=3$.
Thus, to complete the proof, it suffices
to show $rx_3(K_{s,t})\leq 6$, $3\leq s \leq t$.
Let $U$ and $W$ be the two partite sets of $K_{s,t}$,
where $|U|=s$ and $|W|=t$. Suppose
$U=\{u_1,u_2,\cdots, u_s\},~W=\{w_1,w_2,\cdots,w_t\}$.

Clearly we can find a connected 2-dominating set
$D=\{u_1,u_2,w_1,w_2\}$ of $K_{s,t}$. In addition, $K_{s,t}\setminus D$
is connected, $Z=\emptyset$, by
Theorem \ref{thm5}, $rx_3(K_{s,t})\leq rx_3(G[D])+4=6$.

To prove the sharpness of the above upper bound, we derive the following claim.

{\bf Claim.} For any $s\geq 3$, $t\geq 2\times 6^s$, $rx_3(K_{s,t})=6$.

Firstly, we consider the graph $K_{3,t}$. We may assume that there exists a 3-rainbow coloring $c$ of $K_{3,t}$ with $k$ colors.
Corresponding to this 3-rainbow coloring, for every vertex $w$ in $W$, there is a color code, code($w$), assigned $a_i=c(u_iw)\in \{1,2,\cdots,k\}$, $1\leq i\leq 3$. Observe that any three vertices have at least three distinct colors appeared in their color codes. Thus, we know that at most two vertices have the common code except possibly when $a_1\neq a_2\neq a_3$.
Otherwise, there is no rainbow tree containing these three vertices which have
the same code and at most two colors in color code. Therefore, when $t\geq 2k^3$, there must exist three vertices  $w'$, $w''$, $w'''$ such that code ($w'$)=code($w''$)=code($w'''$)=$\{a_1,a_2,a_3\}$
and $a_1\neq a_2\neq a_3$. If a rainbow tree containing $S=\{w',\ w'',\ w'''\}$, it must contain $u_1,u_2,u_3$ and $w_i$ to guarantee  its connectivity, where $w_i$ belongs to
$W$ and code($w_i$)=$\{b_1,b_2,b_3\}$, where $a_i$, $b_j$ are different from each other, $i=1,2,3;\ j=1,2,3$. Thus $k\geq 6$. So $rx_3(K_{3,t})=6$, when $t\geq 2\times 6^3$. Similarly, we can prove $rx_3(K_{s,t})=6$, for $s\geq 4$, $t\geq 2\times 6^s$.
Thus, this claim also provides  the tight proof of the Theorem \ref{thm5}.
\end{proof}

 Here, we can simply check that the upper bound can not  be generalized to the graphs $K_{2,t}$. By the same method used in the above claim, We may assume that there exists a 3-rainbow coloring $c$ of $K_{2,t}$ with $k$ colors.
Corresponding to this 3-rainbow coloring, there is a color code, code($w$), assigned $a_i=c(u_iw)\in \{1,2,\cdots,k\}$ for $i=1,2$. Observe that at most
two vertices have the common code. It follows, $t\leq 2k^2$. Thus, $k$ is not less than a certain constant  when $t$ is enough large.

To state next theorem, we need to make  more definitions.
A graph $G$ is called
{\em a perfect connected dominant graph} if $\gamma(X)=\gamma_c(X)$, for
each connected induced subgraph $X$ of $G$.
 If $G$ and $H$ are two graphs, we say that $G$ is {\em $H$-free} if $H$ does
not appear as an induced subgraph of $G$. Furthermore,
if $G$ is $H_1$-free and $H_2$-free, we say that $G$ is {\em $(H_1,H_2)$-free}.
Next,
we determine the upper bound for $3$-rainbow index of $(P_5,C_5)$-free graphs.

Zverovich has obtained the following result.

\begin{thm}\cite{Zverovich2003domiant}\label{thm11}
A graph $G$ is a perfect connected-dominant graph if and only if
$G$ contains no induced path $P_5$ and  induced cycle $C_5$.
\end{thm}

As shown in Theorem \ref{thm5}, in order to obtain a better bound of $3$-rainbow index, we may turn to  a smallest possible connected $2$-dominating set. For a graph with minimal degree $\delta\geq 3$, B.Reed proved the following conclusion in \cite{B.Reed1996}.

\begin{thm}\cite{B.Reed1996}\label{thm12}
If $G$ is connected graph with $\delta \geq 3$, then
 $\gamma(G)\leq \frac{3n}{8}$.
\end{thm}
For $(P_5,C_5)$-free graphs $\delta \geq 3$, we have $\gamma_c(G)\leq \frac{3n}{8}$.
Inspired by this result, the extension of the idea of connected dominating set to connected $2$-dominating set is what gives the following lemma.

\begin{lem}\label{lem3}
Let $G$ be a connected graph of order $n$ with minimal degree
$\delta\geq 2$. If $D$ is a connected dominating set in a graph $G$, then there is a set of vertices $D'\supseteq D$ such that $D'$ is a connected
2-dominating set and $|D'|\leq \frac{1}{2}n+\frac{1}{2}|D|$.
\end{lem}
\begin{proof}
There are two types of the components of $G\setminus D$: singletons
and connected subgraphs. Let $P$ be the set of the singletons, and $Q$ be the set of the connected components of $G\setminus D$. Note that
$G\setminus D=P\cup Q$. Since $\delta\geq 2$, for any vertex $v$ in $P$, it has at least two neighbors in $D$. In every non-singleton connected
component of $Q$, we choose a spanning tree. This gives a spanning forest on
$V(Q)$. Choose $X$ and $Y$ as any one of
the bipartitions  defined by this forest. Without loss of generality, we suppose that $|X|\leq|Y|$.

Stage ~~$D'=D$

~~~~while  $ \exists v \in V(Q)$ such that $|N(v)\bigcap D|=1$

~~~~$ \{$

~~~~~~~~~If $v \in Y $

~~~~~~~~~~~Pick a vertex $u \in N(v)\bigcap X $.

~~~~~~~~~~~Let $D' =D' \bigcup \{u\} $

~~~~~~~~~else

~~~~~~~~~~~$D' =D'  \bigcup \{v\}$

~~~~$ \}$

Clearly $D'$ remains  to be connected. Since stage ends only
when any vertex in $V(Q)$ has at least 2 neighbors in $D'$. So
the final $D'$ is a connected $2$-dominating  set. Let $k$ be the number of
iterations executed. Since we add a vertex in
$X$ to $D'$, $|X|$ reduces by 1 in every iteration, $k\leq |X|\leq
\frac{1}{2}(n-|D|)$, so $|D'|\leq
|D|+k \leq |D|+\frac{1}{2}(n-|D|)=\frac{1}{2}n+\frac{1}{2}|D|$.
\end{proof}

For a connected $(P_5,C_5)$-free graph $G$ with $\delta \geq 3$, we can derive the
following result by  Theorem \ref{thm5}, Theorem \ref{thm11}, Theorem \ref{thm12} and Lemma \ref{lem3}.

\begin{thm}\label{thm4}
For every connected $(P_5,C_5)$-free graphs $G$ with $\delta(G)\geq 3$,
$rx_3(G)\leq \frac{11}{16}n+3$.
\end{thm}

\begin{proof}
For every connected $(P_5,C_5)$-free graphs $G$ with $\delta(G)\geq 3$, from the Theorem \ref{thm11},
$\gamma(G)=\gamma_c(G)$. And by the Theorem \ref{thm12}, we have $\gamma(G)\leq \frac{3n}{8}$. Thus, $\gamma_c(G)\leq \frac{3n}{8}$. Combining this with Lemma \ref{lem3}, the graph $G$ have a connected $2$-dominating set $D$ with order less than $\frac{11}{16}n$. Observe that the connected $2$-dominating set $D$ can get a 3-rainbow coloring using $|D|-1$ colors by ensuring that every edge of some spanning tree gets distinct color. So the upper bound follows immediately from Theorem \ref{thm5}.
\end{proof}

\section{Upper bounds for $3$-rainbow index of  general graphs}
In this section, we derive a sharp bound for $3$-rainbow index of general graphs by block decomposition. And we also show a better bound for  $3$-rainbow index of general graphs with $\delta(G)\geq 3$ by connected $2$-dominating set.

 Let $\mathcal{A}$ be the set of blocks of $G$, whose element is $K_2$; Let $\mathcal{B}$ be the set of blocks of $G$, whose element is $K_3$; Let $\mathcal{C}$ be the set of blocks of $G$, whose element $X$ is a cycle or a block of order $4\leq |V(X)|\leq 6$; Let $\mathcal{D}$ be the set of blocks of $G$, whose element $X$ is  not a cycle and $|V(X)|\geq 7$.

\begin{thm}\label{thm4}
Let $G$ be a connected graph of order $n~(n \geq 3)$. If $G$ has a
block decomposition $B_1,B_2,\cdots,B_q$, then $rx_3(G)\leq n-|\mathcal{C}|-2|\mathcal{D}|-1$, and the upper bound is tight.
\end{thm}
\begin{proof}
Let $G$ be a connected graph of order $n$ with $q$ blocks in
its block decomposition. If $q=1$, then we have done by  Theorem \ref{thm0} and $rx_3(K_3)=2$, which satisfies the above bound. Thus, we suppose $q\geq 2$.

Note that $|\mathcal{A}\cup \mathcal{B}\cup \mathcal{C}\cup \mathcal{D}|=q$.
From the Theorem \ref{thm0}, we get $rx_3(X)\leq |X|-2$ for $X\in \mathcal{C} $ and
$rx_3(X)\leq |X|-3$ for $X\in \mathcal{D}$.
Hence,
it follows that
\begin{eqnarray*}
rx_3(G)&\leq &\sum_{X\in \mathcal{A}}1 +\sum_{X\in \mathcal{B}}2+\sum_{X\in \mathcal{C}}(|X|-2)+\sum_{X\in \mathcal{D}}(|X|-3)\\
&=&n-|\mathcal{C}|-2|\mathcal{D}|-1.
\end{eqnarray*}

In order to prove that the upper bound is tight, we construct
the graph $G$ of order  $n$, as shown in Figure 1,
consisting
of $(n-3r-7)$ $K_2$, $r$ cycles of order 4 and one 7-length-cycle
with a chord.
It is clear that   $|\mathcal{C}|$=$r$, $|\mathcal{D}|=1$. We consider the size of a rainbow tree $T$ contain the vertices $\{u,v,w\}$. $|E(T)|=n-4r-7+3r+4=n-r-3$ and
$rx_3(G)\leq n-|\mathcal{C}|-2|\mathcal{D}|-1=n-r-3$ by the above theorem.
we have $rx_3(G)=n-|\mathcal{C}|-2|\mathcal{D}|-1$.
\begin{figure}[h,t,b,p]
\begin{center}
\includegraphics[width=10cm]{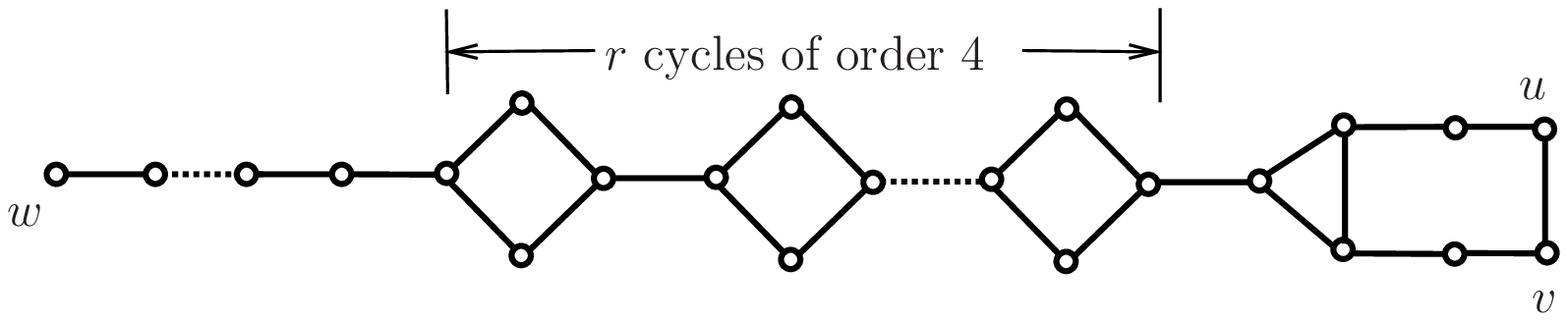}\\
Figure 1: Graph for Theorem \ref{thm4}
\end{center} \label{fig8}
\end{figure}
\end{proof}

We finish this section with  general graphs with minimal degree at least $3$. Here, we denote as $q_{max}(G)$  the maximum number of components of $G\backslash u$ among all vertices $u\in V$. The following result is needed in the sequel.

\begin{thm}\cite{AD}\label{thm13}
 Let G be a connected graph on $n$ vertices with minimum degree $\delta\geq 2$ and let $k$ be an integer with $1 \leq k\leq \delta$. Then $\gamma_k^c\leq n-q_{max}(G)(\delta-k+1)$
\end {thm}

For general graphs with  $\delta \geq 3$, we obtain an upper bound for $3$-rainbow index  from Theorem \ref{thm5} and Theorem \ref{thm13}.
\begin{thm}\label{thm14}
Let $G$ be a connected graph with minimal degree
$\delta\geq 3$. Then $rx_3(G)\leq n-q_{max}(\delta-1)+3$.
\end{thm}
Note that the bound of $3$-rainbow index is better for the graphs with cut vertices and larger minimal degree.

\end{document}